\documentclass{amsart}
\usepackage[T1]{fontenc}
\usepackage[utf8x]{inputenc}

\usepackage{amscd,amsmath,amssymb,amsfonts,epsfig,graphics,amsthm}
\usepackage{booktabs}
\usepackage{graphicx}
\usepackage{epstopdf}
\usepackage{mathabx} 
\usepackage{nccmath} 

\usepackage{hyperref}
\usepackage[authoryear]{natbib}
\usepackage{xcolor}

\let\oldcite\cite
\renewcommand{\cite}[1]{{\color{red}\oldcite{#1}}}

\let\oldcitep\citep
\renewcommand{\citep}[1]{{\color{red}\oldcitep{#1}}}

\let\oldcitet\citet
\renewcommand{\citet}[1]{{\color{red}\oldcitet{#1}}}
\usepackage{mdframed} 
\usepackage{lipsum}
\usepackage{tcolorbox} 
\usepackage{comment}  

\tcbuselibrary{theorems}
\usepackage{MnSymbol}

\usepackage[margin=3cm]{geometry}
\newtheorem{theorem}{Theorem}[section]
\vspace{10pt}

\newtheorem{lemma}[theorem]{Lemma}

\theoremstyle{definition}

\vspace{10pt}

\newtheorem{example}[theorem]{Example}
\theoremstyle{remark}
\newtheorem{remark}[theorem]{Remark}

\newtheorem{corollary}[theorem]{Corollary}

\newtcbtheorem{mytheo}{Theorem}%
{colback=green!5,colframe=black!35!black,fonttitle=\bfseries}{th}

\newcommand{\nd}{\noindent}
\newcommand{\norm}[1]{\left\lVert#1\right\rVert}

\newcommand{\Ddim}{\mathbb{R}^D}



\begin{document}
\pagenumbering{arabic}

\title[Identifiability properties of the unnormalized graph Laplace operators]{\large \textbf{Identifiability of the unnormalized graph Laplace operators}}
\date{\today}

\maketitle
\begin{center}
  \textbf{Susovan Pal} \\
  Department of Mathematics and Data Science, Vrije Universiteit Brussel (VUB) \\
  Pleinlaan 2, B-1050 Elsene/Ixelles, Belgium \\
  \texttt{susovan.pal@vub.be, susovan97@gmail.com}
  
\end{center}
\begin{abstract}
 In this short note, we show that the continuous intrinsic graph Laplace operator with Gaussian kernel on a compact Riemannian manifold without boundary uniquely determines both the Riemannian metric and the sampling density, provided the latter is positive. In contrast, the corresponding continuous extrinsic graph Laplace operator uniquely determines the sampling measure; moreover, when the operator is defined via an embedding into Euclidean space, it also uniquely determines the induced Riemannian metric and the sampling density.
\end{abstract}

\nd \\
\smallskip
\nd \textbf{Word count:} 4500 (approx.)

\section{Introduction}

\nd Let $(M,g)$ be a smooth compact $d$-dimensional Riemannian manifold without boundary, and let
$\mu=\mathrm{dvol}_g$ denote the volume form. Let $p$ be a continuous probability density on $M$ with respect to
$\mu$, and let $(X_i)_{i=1}^n$ be i.i.d.\ $M$--valued samples drawn from $p$. Given a bandwidth $t>0$,
we define the (unnormalized) \textit{discrete graph Laplace operator} by
\begin{equation}\label{eq:discrete_GL}
(L_{n,t} f)(x)
:= \frac{1}{n\,t^{d/2+1}} \sum_{j=1}^n
e^{-\frac{\ell(x,X_j)^2}{t}}\big(f(x)-f(X_j)\big), f \in C(M)
\end{equation}
where $\ell$ is a distance function on $M$ specified below.

\nd As $n\to\infty$, the operator $L_{n,t}$ converges almost surely, pointwise in $x$, to the
corresponding \textit{continuous} integral operator
\begin{equation}\label{eq:continuous_GL}
L_t f(x)
= \frac{1}{t^{d/2+1}} \int_M
  e^{-\frac{\ell(x,y)^2}{t}}\big(f(x)-f(y)\big) p(y)\,d\mu(y),  f \in C(M)
\end{equation}
which we call the \textit{unnormalized continuous graph Laplace operator} associated to $(M,g,p, \ell)$.

\nd We call $L_t$ the continuous \textit{intrinsic} graph Laplace operator if $\ell:=d_g$, the intrinsic
distance induced by $g$. In this case, we denote the operator by $L^{(g)}_{t,p}$. For practical
purposes, $d_g$ may be difficult to estimate, so in the manifold learning community one often works
with the \textit{extrinsic} continuous graph Laplace operator, where
$\ell(x,y):=\norm{\iota(x)-\iota(y)}_{\Ddim}$ is the ambient/extrinsic distance for a smooth embedding
$\iota:M\to {\Ddim}$. We denote this operator by $L^{(ext)}_{t,p}(g;\iota).$ Of special importance is the case where $g$ is induced by $\iota.$\\

\nd Both the discrete and continuous operators are extensively used for manifold learning, see e.g. \cite{BelkinNiyogiJCSS2008}, \cite{BelkinNiyogiNIPS2006}. In the boundaryless case, the pointwise and spectral convergences of the discrete and continuous operators to the weighted Laplace-Beltrami operator for $t\downarrow 0$ have been exclusively studied in above as well as \cite{HeinAudibertLuxburgJMLR2007}, \cite{HeinAudibertLuxburgNIPS2005}. More recently, these convergence questions for manifolds with smooth and non-smooth boundaries and with isolated singularities have been addressed in \citet{Belkin2012Boundary}, \cite{SingerWu2012}, \cite{SingerWu2017}, \cite{PalTewodrose2025}, \cite{Pal2025_AsymptoticsGraphLaplace_IsolatedSingularity}.\\

\nd While the limit $t\downarrow 0$ explains how graph Laplacians approximate \textit{intrinsic }differential operators, a distinct and more
identifiability-oriented question is the inverse one:
\emph{to what extent does either a single member or the entire family of graph Laplace operators (as $t$ is kept fixed or varies) determine the underlying
geometric and probabilistic structure?}
In this note we address this question for the \textit{continuous} operators on compact manifolds without boundary.
Under mild smoothness assumptions, we establish identifiability results for (i) the Riemannian metric $g$ (Theorem \ref{thm:intrinsic_metric_identifiable}) and
(ii) the sampling density $p$ (Theorem \ref{thm:intrinsic_density_identifiable}), and their joint identification (Theorem \ref{thm:intrinsic_joint_identifiable}) from the knowledge of the operator family in the \textit{intrinsic} setting (cf. Section \ref{scn:intrinsic-GL}), and we delineate
what can and cannot be recovered in the\textit{ extrinsic} setting (cf. Section \ref{scn:extrinsic-GL}), Theorems \ref{thm:extrinsic_identifies_measure}, \ref{thm:extrinsic_induced_metric_identifiable}, \ref{thm:extrinsic_joint_identifies_metric_and_density}.
To the best of our knowledge, such identifiability statements are not stated explicitly in the existing graph Laplacian literature, even though their proofs rely on elementary consequences of the integral representations.\\

\nd Although only discrete graph Laplace operators are available in applications, classical consistency results mentioned before show that they converge to the corresponding continuous operators as the sample size grows. Our identifiability results therefore clarify that there is a \textit{unique Riemannian metric} \textit{and}  \textit{a unique density} to be recovered in this limit.\\

\noindent Similar to discrete \textit{graph} Laplace operators, discrete Laplace-Beltrami operators have long been
a central object in \textit{discrete differential geometry} and \textit{geometry processing},
where they are defined on weighted graphs, triangulated manifolds, meshes, and simplicial complexes,
and serve as discrete \emph{deterministic} analogues of the Laplace--Beltrami operator
\cite{BobenkoSuris2008}. In contrast, discrete \textit{graph} Laplace operators arising in manifold
learning serve as \emph{random, sample-dependent} analogues of \textit{weighted}
Laplace--Beltrami operators. Beyond approximation and convergence questions, discrete Laplace-Beltrami operators have also been studied
from an inverse or identifiability perspective. In particular, \cite{GuGuoLuoZeng2010} showed that
the cotangent (discrete Laplace--Beltrami) operator on a triangulated \textit{surface} uniquely
determines the underlying discrete Riemannian metric (up to a global scaling), via a variational
formulation. A natural conjecture extending this result to arbitrary-dimensional polyhedral
manifolds remains open.

These results suggest that it is natural to ask analogous identifiability questions for
kernel-based graph Laplace operators arising in manifold learning, which is the focus of the
present work. Finally, we emphasize that identifiability results for graph Laplace operators do not automatically translate to the same for discrete Laplace-Beltrami operators.


\section{Identifiability of intrinsic continuous graph Laplacians}\label{scn:intrinsic-GL}

\begin{lemma}[\textbf{A $C^2$ Riemannian metric is determined by its distance}]
\label{lem:distance_determines_metric}
Let $M$ be a smooth manifold and let $g_1,g_2$ be $C^2$ Riemannian metrics on $M$ with
geodesic distance functions $d_{g_1},d_{g_2}$. If $d_{g_1}(x,y)=d_{g_2}(x,y)$ for all
$x,y\in M$, then $g_1=g_2$.
\end{lemma}

\begin{proof}
Fix $x\in M$ and choose local coordinates $(x^1,\dots,x^d)$ on a neighborhood $U$ of $x$.
For $i=1,2$ define the squared distance function
\[
\sigma_i(u,v):=d_{g_i}(u,v)^2,\qquad (u,v)\in U\times U.
\]
For $C^2$ metrics, $\sigma_i$ is $C^2$ in a neighborhood of the diagonal in $U\times U$
(away from the cut locus; in particular, for $u,v$ sufficiently close to $x$).
In these coordinates, the metric tensor is recovered from $\sigma_i$ by the standard identity
\[
(g_i)_{jk}(x)\;=\;-\frac{1}{2}\,\frac{\partial^2}{\partial u^j\,\partial v^k}\,\sigma_i(u,v)\Big|_{u=v=x}.
\]

If $d_{g_1}\equiv d_{g_2}$, then $\sigma_1\equiv \sigma_2$ on $U\times U$, hence their mixed
second derivatives at $(x,x)$ coincide, so $(g_1)_{jk}(x)=(g_2)_{jk}(x)$ for all $j,k$.
Since $x$ is arbitrary, $g_1=g_2$.
\end{proof}

\begin{theorem}[\textbf{Intrinsic operator identifies the Riemannian metric (for fixed bandwidth and positive density)}]
\label{thm:intrinsic_metric_identifiable}
Let $M$ be a smooth compact manifold without boundary, $\dim M=d$.
Fix $t>0$ and $p\in C(M)$ with $p(x)>0$ for all $x\in M$.
For any $C^2$ Riemannian metric $g$ on $M$, define
\[
(L^{(g)}_{t,p} f)(x)
:=\frac{1}{t^{\frac d2+1}}\int_M \exp\!\Big(-\frac{d_g(x,y)^2}{t}\Big)\,\bigl(f(x)-f(y)\bigr)\,p(y)\,d\mu_g(y),
\qquad f\in C(M).
\]
If $g_1,g_2$ are $C^2$ metrics on $M$ such that
\[
L^{(g_1)}_{t,p}=L^{(g_2)}_{t,p}\quad\text{as operators on }C(M),
\]
then $g_1=g_2$.
\end{theorem}

\begin{proof}
Write $\mu_i:=\mu_{g_i}$ and $K_i(x,y):=\exp\!\big(-d_{g_i}(x,y)^2/t\big)$.
For $i=1,2$ define the integral operators
\[
(T_i f)(x):=\int_M K_i(x,y)\,f(y)\,p(y)\,d\mu_i(y),
\qquad f\in C(M).
\]
Then
\[
(L^{(g_i)}_{t,p} f)(x)=\frac{1}{t^{\frac d2+1}}\Big( (T_i \mathbf 1)(x)\,f(x) - (T_i f)(x)\Big).
\]
Assuming $L^{(g_1)}_{t,p}=L^{(g_2)}_{t,p}$, we obtain for all $f\in C(M)$ and all $x\in M$,
\begin{equation}
\label{eq:mult_identity}
(T_1 f)(x)-(T_2 f)(x)=\big((T_1\mathbf 1)(x)-(T_2\mathbf 1)(x)\big)\,f(x).
\end{equation}
Define $a(x):=(T_1\mathbf 1)(x)-(T_2\mathbf 1)(x)$.

\smallskip
\noindent\textbf{Step 1: $T_1\mathbf 1=T_2\mathbf 1$ and hence $T_1=T_2$.}
Fix $x\in M$ and define the Borel measures
\[
m_{i,x}(B):=\int_B K_i(x,y)\,p(y)\,d\mu_i(y),
\qquad B\subset M\ \text{Borel}.
\]
Then $(T_i f)(x)=\int_M f(y)\,dm_{i,x}(y)$ for all $f\in C(M)$, so \eqref{eq:mult_identity} becomes
\[
\int_M f(y)\,d(m_{1,x}-m_{2,x})(y)=a(x)\,f(x)
\qquad(\forall f\in C(M)).
\]
By the Riesz representation theorem, the functional $f\mapsto a(x)f(x)$ is represented by the
measure $a(x)\delta_x$, where $\delta_x$ denotes the Dirac mass at $x,$ hence
\begin{equation}
\label{eq:dirac_measure_identity}
m_{1,x}-m_{2,x}=a(x)\delta_x
\qquad\text{as finite signed measures.}
\end{equation}
Evaluating \eqref{eq:dirac_measure_identity} on $\{x\}$ yields
\[
(m_{1,x}-m_{2,x})(\{x\})=a(x).
\]
On the other hand,
\[
m_{i,x}(\{x\})=\int_{\{x\}} K_i(x,y)\,p(y)\,d\mu_i(y)=0
\]
since $\mu_i(\{x\})=0$ for the Riemannian volume measure $\mu_i$.
Thus $(m_{1,x}-m_{2,x})(\{x\})=0$, hence $a(x)=0$.
Since $x$ was arbitrary, $a\equiv 0$, so $T_1\mathbf 1=T_2\mathbf 1$ and \eqref{eq:mult_identity}
implies $T_1=T_2$.

\smallskip
\noindent\textbf{Step 2: identification of the kernel and the distance.}
Fix $x\in M$. From $T_1=T_2$ we have, for all $f\in C(M)$,
\[
\int_M f(y)\,K_1(x,y)\,p(y)\,d\mu_1(y)=\int_M f(y)\,K_2(x,y)\,p(y)\,d\mu_2(y).
\]
Since $\mu_1$ and $\mu_2$ are Riemannian volume measures on $M$,  they are \textit{mutually absolutely
continuous}; write $d\mu_2=w\,d\mu_1$ for a continuous $w>0$. Then
\[
\int_M f(y)\,\big(K_1(x,y)-K_2(x,y)\,w(y)\big)\,p(y)\,d\mu_1(y)=0
\qquad(\forall f\in C(M)).
\]
Hence
\[
\big(K_1(x,y)-K_2(x,y)\,w(y)\big)\,p(y)=0
\quad\text{for $\mu_1$-a.e.\ }y.
\]
Using $p>0$, this gives us
\[
K_1(x,y)-K_2(x,y)\,w(y)=0
\quad\text{for $\mu_1$-a.e.\ }y.
\]
The map $y\mapsto K_1(x,y)-K_2(x,y)\,w(y)$ is continuous since $K_i, w$ are so, and $\mu_1$ has full support,
so the equality holds for all $y\in M$. Setting $y=x$ gives $1=w(x)$, hence $w\equiv 1$ and
$\mu_1=\mu_2$. Therefore $K_1(x,y)=K_2(x,y)$ for all $x,y$, i.e.\ $d_{g_1}=d_{g_2}$.
By Lemma~\ref{lem:distance_determines_metric}, $g_1=g_2$.
\end{proof}

\begin{theorem}[\textbf{Intrinsic operator identifies the positive density (for fixed bandwidth and $C^2$ Riemannian metric)}]
\label{thm:intrinsic_density_identifiable}
Let $M$ be compact without boundary and let $g$ be a $C^2$ Riemannian metric on $M$.
Fix $t>0$. For $p_j, j=1, 2\in C(M)$ with $p_j>0$ define $L^{(g)}_{t,p_j}$ by
\[
(L^{(g)}_{t,p_j} f)(x)
:=\frac{1}{t^{\frac d2+1}}\int_M \exp\!\Big(-\frac{d_g(x,y)^2}{t}\Big)\,\bigl(f(x)-f(y)\bigr)\,p_j(y)\,d\mu_g(y).
\]
If $p_1,p_2\in C(M)$ are strictly positive and
\[
L^{(g)}_{t,p_1}=L^{(g)}_{t,p_2}\quad\text{as operators on }C(M),
\]
then $p_1=p_2$.
\end{theorem}

\begin{proof}
Let $\mu:=\mu_g$ and $K(x,y):=\exp(-d_g(x,y)^2/t)$.
As in \textbf{Step 1} of Theorem~\ref{thm:intrinsic_metric_identifiable}, operator equality implies equality of the
integral operators
\[
(T_{p_j}f)(x):=\int_M K(x,y)\,f(y)\,p_j(y)\,d\mu(y)\qquad (j=1,2).
\]
Fix $x$. Then $T_{p_1}=T_{p_2}$ implies 
\[
K(x,y)\,p_1(y)=K(x,y)\,p_2(y)\quad\text{for $\mu$-a.e.\ }y.
\]
Since $Kp_j, j=1, 2$ are continuous in $y$, and $\mu$ has full support on all of $M$, and $K(x,y)>0$, we conclude $p_1(y)=p_2(y)$ for all $y$.
\end{proof}

\begin{theorem}[\textbf{Intrinsic operator identifies the pair $(g,p)$ (for fixed $t$)}]
\label{thm:intrinsic_joint_identifiable}
Let $M$ be compact without boundary. Fix $t>0$.
Let $g_1,g_2$ be $C^2$ Riemannian metrics on $M$, and let $p_1,p_2\in C(M)$ satisfy $p_1>0$ and $p_2>0$.
Define $L^{(g_i)}_{t,p_i}$ as above.
If
\[
L^{(g_1)}_{t,p_1}=L^{(g_2)}_{t,p_2}\quad\text{as operators on }C(M),
\]
then $g_1=g_2$ and $p_1=p_2$.
\end{theorem}

\begin{proof}
Let $\mu_i:=\mu_{g_i}$ and $K_i(x,y):=\exp(-d_{g_i}(x,y)^2/t)$.
As in \textbf{Step 1 }of Theorem~\ref{thm:intrinsic_metric_identifiable}, operator equality implies equality of the
integral operators
\[
(T_i f)(x):=\int_M K_i(x,y)\,f(y)\,p_i(y)\,d\mu_i(y),
\]

\nd Write $d\mu_2=w\,d\mu_1$ with $w\in C(M)$ and $w>0$.
Arguing exactly as in Theorem~\ref{thm:intrinsic_metric_identifiable} (Radon--Nikodym + full support on all of $M$),
we obtain the pointwise identity
\[
K_1(x,y)\,p_1(y)=K_2(x,y)\,p_2(y)\,w(y)\qquad(\forall x,y\in M).
\]
Setting $x=y$ gives $p_1(y)=p_2(y)\,w(y)$ for all $y$, hence $w=p_1/p_2$.
Substituting back yields
\[
K_1(x,y)\,p_1(y)=K_2(x,y)\,p_2(y)\,\frac{p_1(y)}{p_2(y)}=K_2(x,y)\,p_1(y),
\]
so (since $p_1>0$) $K_1(x,y)=K_2(x,y)$ for all $x,y$, i.e. $d_{g_1}=d_{g_2}$.
By Lemma~\ref{lem:distance_determines_metric}, $g_1=g_2$, hence $\mu_1=\mu_2$ and $w\equiv 1$.
From $p_1=p_2\,w$ we conclude $p_1=p_2$.
\end{proof}

\section{On identifiability of extrinsic/ambient continuous graph Laplacians}\label{scn:extrinsic-GL}

\nd Let $\iota:M\to\mathbb{R}^D$ be a smooth embedding. Fix $t>0.$
For a $C^2$ Riemannian metric $g$ on $M$ and $p\in C(M)$ with $p>0$, recall the definition of the extrinsic continuous graph Laplacians:
\begin{equation}
\label{eq:extrinsic_L}
(L^{\mathrm{ext}}_{t,p}(g;\iota)f)(x)
:=\frac{1}{t^{\frac d2+1}}\int_M \exp\!\Big(-\frac{\|\iota(x)-\iota(y)\|^2}{t}\Big)\,\bigl(f(x)-f(y)\bigr)\,p(y)\,d\mu_g(y),
\qquad f\in C(M).
\end{equation}

\nd Note that we do not assume \textit{at this point} that the Riemannian metric $g$ is induced by $\iota.$
Set
\[
K_\iota(x,y):=\exp\!\Big(-\frac{\|\iota(x)-\iota(y)\|^2}{t}\Big)
\quad\text{and}\quad
m:=p\,\mu_g.
\]

\begin{theorem}[\textbf{Extrinsic operator identifies the sampling measure $m=p\,\mu_g$}]
\label{thm:extrinsic_identifies_measure}
Fix $t>0$ and a smooth embedding $\iota:M\to\mathbb{R}^D$.
Let $(g_1,p_1)$ and $(g_2,p_2)$ be such that each Riemannian metric $g_i$ is $C^2$ and each density function $p_i\in C(M)$ satisfies $p_i>0$.
Let $m_i:=p_i\,\mu_{g_i}$. Assume
\[
L^{\mathrm{ext}}_{t,p_1}(g_1;\iota)=L^{\mathrm{ext}}_{t,p_2}(g_2;\iota)
\quad\text{as operators on }C(M).
\]
Then $m_1=m_2$.
\end{theorem}

\begin{proof}
Define the integral operators
\[
(T_i f)(x):=\int_M K_\iota(x,y)\,f(y)\,dm_i(y),\qquad f\in C(M), \, m_i:=p_i\mu_{g_i}
\]
so that just like in the proof of Theorem \ref{thm:intrinsic_metric_identifiable} expanding \eqref{eq:extrinsic_L} yields
\[
(L^{\mathrm{ext}}_{t,p_i}(g_i;\iota)f)(x)
=\frac{1}{t^{\frac d2+1}}\Big((T_i\mathbf 1)(x)\,f(x)-(T_i f)(x)\Big).
\]
From operator equality we get, for all $f\in C(M)$,
\[
(T_1 f)(x)-(T_2 f)(x)=\big((T_1\mathbf 1)(x)-(T_2\mathbf 1)(x)\big)\,f(x).
\]
\nd Now following \textbf{Step 1}  (the exact same logic (Riesz representation plus Dirac mass argument) of Theorem \ref{thm:intrinsic_metric_identifiable}, we conclude $T_1=T_2,$ which implies by Riesz representation theorem that $K_{\iota}(x,.)m_1= K_{\iota}(x,.)m_2$ as measures. Next, $y\mapsto K_\iota(x,y)$ is continuous and strictly positive on $M$.
Since $\mu_{g_1}$ and $\mu_{g_2}$ are mutually absolutely continuous, we may write
\(
dm_i(y) = q_i(y)\, d\mu_{g_1}(y), q_i \in C(M), \ q_i>0,\ i=1,2.
\)
From the equality
\(
K_\iota(x,\cdot)\, m_1 = K_\iota(x,\cdot)\, m_2
\),
it follows that
\(
K_\iota(x,y)\, q_1(y) = K_\iota(x,y)\, q_2(y)
\) for $\mu_{g_1}$ a.e. $y\in M$. Since $K_\iota(x,y)>0$ for all $y\in M$, we conclude that $q_1(y)=q_2(y)$ for
$\mu_{g_1}$-almost every $y$, and hence $m_1=m_2$. 
\end{proof}

\begin{corollary}[\textbf{Fixed embedding, fixed density $p$: extrinsic operator identifies the volume form $\mu_g$}]
\label{cor:extrinsic_identifies_volume}
Fix $t>0$, a smooth embedding $\iota$, and a continuous $p>0$.
If $g_1,g_2$ are $C^2$ metrics such that
\[
L^{\mathrm{ext}}_{t,p}(g_1;\iota)=L^{\mathrm{ext}}_{t,p}(g_2;\iota)\quad\text{on }C(M),
\]
then $\mu_{g_1}=\mu_{g_2}$.
\end{corollary}

\begin{proof}
By Theorem~\ref{thm:extrinsic_identifies_measure}, we have $p\,\mu_{g_1}=p\,\mu_{g_2}$.
Since $p>0$ is continuous on $M$, therefore 
$\mu_{g_1}=\mu_{g_2}$.
\end{proof}

\begin{corollary}[\textbf{Fixed embedding, fixed metric: extrinsic operator identifies density}]
\label{cor:extrinsic_identifies_density}
Fix $t>0$, a smooth embedding $\iota$, and a $C^2$ metric $g$.
If $p_1,p_2\in C(M)$ satisfy $p_1>0$, $p_2>0$, and
\[
L^{\mathrm{ext}}_{t,p_1}(g;\iota)=L^{\mathrm{ext}}_{t,p_2}(g;\iota)\quad\text{on }C(M),
\]
then $p_1=p_2$.
\end{corollary}

\begin{proof}
By Theorem~\ref{thm:extrinsic_identifies_measure}, $p_1\,\mu_g=p_2\,\mu_g$.
Since $\mu_g$ has full support, the continuous function $p_1-p_2$ must vanish everywhere, hence $p_1=p_2$.
\end{proof}

\begin{remark}[\textbf{Knowing volume form $\mu_g$ does not determine metric (fixed embedding)}]
\label{rem:extrinsic_metric_not_identifiable}
Corollary~\ref{cor:extrinsic_identifies_volume} shows that for fixed $(t,\iota,p)$ the operator
$L^{\mathrm{ext}}_{t,p}(g;\iota)$ determines $\mu_g$. In general, $\mu_g$ does not determine $g$,
so the metric is not identifiable from $L^{\mathrm{ext}}_{t,p}(g;\iota)$.
\end{remark}

\begin{example}[\textbf{Explicit counterexample on the torus $\mathbb{T}^2$}]
\label{ex:torus_volume_same_metric_diff}
Let $M=\mathbb{T}^2=\mathbb{R}^2/(2\pi\mathbb{Z})^2$ with coordinates $(x,y)$.
Fix $\iota:\mathbb{T}^2\to\mathbb{R}^4$ by
\[
\iota(x,y)=(\cos x,\sin x,\cos y,\sin y).
\]
Fix any $t>0$ and any $p\in C(\mathbb{T}^2)$ with $p>0$.
Let
\[
g_1:=dx^2+dy^2,
\qquad
g_2:=a^2\,dx^2+a^{-2}\,dy^2\quad\text{for some constant }a\neq 1.
\]
Then $\det(g_1)=1$ and $\det(g_2)=a^2\cdot a^{-2}=1$, hence $\mu_{g_1}=\mu_{g_2}$ and therefore
$L^{\mathrm{ext}}_{t,p}(g_1;\iota)=L^{\mathrm{ext}}_{t,p}(g_2;\iota)$, while $g_1\neq g_2$.
\end{example}

\subsection{Identifiability in the case of embedding-induced Riemannian metric}
\nd Next we assume that the Riemannian metric is induced by the embedding, the extrinsic graph Laplace operator also uniquely identifies the induced Riemannian metric. 
\begin{theorem}[\textbf{Extrinsic operator identifies the embedding-induced Riemannian metric for fixed bandwidth}]
\label{thm:extrinsic_induced_metric_identifiable}
Let $M$ be a compact smooth manifold without boundary, $\dim M=d$.
Fix $t>0$ and a continuous density $p(x)>0$ for all $x\in M$.
Let $\iota_1,\iota_2:M\to {\Ddim}$ be $C^2$ embeddings and denote the induced metrics
\[
g_i:=\iota_i^*\langle\cdot,\cdot\rangle_{{\Ddim}},
\qquad \mu_i:=\mu_{g_i},
\qquad 
K_i(x,y):=\exp\!\Big(-\frac{\|\iota_i(x)-\iota_i(y)\|^2}{t}\Big),
\quad i=1,2.
\]
Define the corresponding extrinsic continuous graph Laplace operators by
\[
(L^{\mathrm{ext}}_{t,p}(g_i;\iota_i)f)(x)
:=\frac{1}{t^{\frac d2+1}}
\int_M K_i(x,y)\,\bigl(f(x)-f(y)\bigr)\,p(y)\,d\mu_i(y),
\qquad f\in C(M).
\]
If
\[
L^{\mathrm{ext}}_{t,p}(g_1;\iota_1)=L^{\mathrm{ext}}_{t,p}(g_2;\iota_2)
\quad\text{as operators on }C(M),
\]
then
\[
\|\iota_1(x)-\iota_1(y)\|
=\|\iota_2(x)-\iota_2(y)\|
\quad\text{for all }x,y\in M,
\]
and consequently $g_1=g_2$.
\end{theorem}

\begin{proof}
For $i=1,2$ define the integral operators
\[
(T_i f)(x):=\int_M K_i(x,y)\,f(y)\,p(y)\,d\mu_i(y),
\qquad f\in C(M).
\]
Then
\[
(L^{\mathrm{ext}}_{t,p}(g_i;\iota_i)f)(x)
=\frac{1}{t^{\frac d2+1}}
\Big((T_i\mathbf 1)(x)\,f(x)-(T_i f)(x)\Big).
\]
Hence the assumption
$L^{\mathrm{ext}}_{t,p}(g_1;\iota_1)=L^{\mathrm{ext}}_{t,p}(g_2;\iota_2)$
implies that for all $f\in C(M)$ and all $x\in M$,
\begin{equation}
\label{eq:ext_mult_identity_induced_rewrite}
(T_1 f)(x)-(T_2 f)(x)
=\big((T_1\mathbf 1)(x)-(T_2\mathbf 1)(x)\big)\,f(x).
\end{equation}

\smallskip
\noindent\textbf{Step 1: reduction to equality of the integral operators.}
The argument showing from \eqref{eq:ext_mult_identity_induced_rewrite} that
\[
T_1\mathbf 1=T_2\mathbf 1
\quad\text{and hence}\quad
T_1=T_2
\]
is \emph{identical} to \textbf{Step~1} in the proof of
Theorem~\ref{thm:intrinsic_metric_identifiable}
(the Riesz representation / Dirac mass argument using non-atomicity of the
Riemannian volume measures), and is therefore omitted.

\smallskip
\noindent\textbf{Step 2: identification of the distance in the ambient space.}
Fix $x\in M$. Since $T_1=T_2$, for all $f\in C(M)$,
\[
\int_M K_1(x,y) f(y)\,p(y)\,d\mu_1(y)
=
\int_M K_2(x,y) f(y)\,p(y)\,d\mu_2(y).
\]
Because $g_1$ and $g_2$ are smooth Riemannian metrics on compact $M$,
their volume measures are mutually absolutely continuous; write
$d\mu_2=w\,d\mu_1$ with a continuous function $w>0$.
Then
\[
\int_M f(y)\,\bigl(K_1(x,y)-K_2(x,y)\,w(y)\bigr)\,p(y)\,d\mu_1(y)=0
\qquad(\forall f\in C(M)).
\]
Since $p>0$ and $\mu_1$ has full support, this implies, using continuity of $K_i, w, i=1,2$
\[
K_1(x,y)=K_2(x,y)\,w(y)\qquad\text{for all }y\in M.
\]
Setting $y=x$ yields $1=1\cdot w(x)$, hence $w(x)=1$ for all $x$ and thus
$w\equiv 1$. Consequently $K_1\equiv K_2$ on $M\times M$, and this gives 
\(
\|\iota_1(x)-\iota_1(y)\|^2
=\|\iota_2(x)-\iota_2(y)\|^2
\text{ for all }x,y\in M.
\)

\smallskip
\noindent\textbf{Step 3: identification of the induced metric.}
Fix $x\in M$ and $v\in T_xM$. Let $\gamma:(-\varepsilon,\varepsilon)\to M$
be a $C^2$ curve with $\gamma(0)=x$ and $\dot\gamma(0)=v$.
For $i=1,2$, Taylor expansion yields
\[
\iota_i(\gamma(s))
=\iota_i(x)+s\,d\iota_i(x)[v]+o(s)\qquad(s\to 0),
\]
and therefore
\[
\|\iota_i(\gamma(s))-\iota_i(x)\|^2
=s^2\,\|d\iota_i(x)[v]\|^2+o(s^2)\qquad(s\to 0).
\]
By Step~2 the left-hand sides coincide for $i=1,2$ for all sufficiently small
$s$, hence
\[
\|d\iota_1(x)[v]\|^2=\|d\iota_2(x)[v]\|^2.
\]
Since $g_i(x)(v,v)=\|d\iota_i(x)[v]\|^2$, this holds for all $x$ and $v$,
and we conclude that $g_1=g_2$.
\end{proof}

\begin{theorem}[\textbf{Fixed embedding: extrinsic operator identifies the density}]
\label{thm:extrinsic_fixed_embedding_identifies_density}
In our setup, fix $t>0$, a $C^2$ embedding $\iota:M\to\mathbb{R}^{D}$ and induced metric $g$.
Let $p_1,p_2\in C(M)$ satisfy $p_i>0$ on $M$ and define $m_i:=p_i\,\mu_g$.
Assume
\(
L^{\mathrm{ext}}_{t,p_1}(g;\iota)=L^{\mathrm{ext}}_{t,p_2}(g;\iota)
\text{ as operators on }C(M)
\). Then $p_1=p_2$ on $M$.
\end{theorem}

\begin{proof}
Immediate.
\end{proof}

\begin{theorem}[\textbf{Extrinsic operator jointly identifies the sampling density and the embedding-induced metric}]
\label{thm:extrinsic_joint_identifies_metric_and_density}
In our setup, fix $t>0$. For $i=1,2$, let $\iota_i:M\to\mathbb{R}^{D}$ be $C^2$ embeddings and denote the induced metrics
\(
g_i:=\iota_i^*\langle\cdot,\cdot\rangle_{\mathbb{R}^{D}},
 \mu_i:=\mu_{g_i},
 K_i(x,y):=\exp\!\Big(-\frac{\|\iota_i(x)-\iota_i(y)\|^2}{t}\Big).
\)
Let $p_1,p_2\in C(M)$ satisfy $p_i>0$ on $M$, and define $m_i:=p_i\,\mu_i$.
Assume
\(
L^{\mathrm{ext}}_{t,p_1}(g_1;\iota_1)=L^{\mathrm{ext}}_{t,p_2}(g_2;\iota_2)
\text{ as operators on }C(M),
\)
where
\(
L^{\mathrm{ext}}_{t,p_i}(g_i;\iota_i)(f)(x)
:=\frac{1}{t^{\frac d2+1}}\int_M K_i(x,y)\,\bigl(f(x)-f(y)\bigr)\,p_i(y)\,d\mu_i(y).
\)
Then
\(
\|\iota_1(x)-\iota_1(y)\|=\|\iota_2(x)-\iota_2(y)\|
\quad\text{for all }x,y\in M,
\)
and consequently $g_1=g_2$, $\mu_1=\mu_2$, and $p_1=p_2$ on $M$.
\end{theorem}

\begin{proof}
Fix $x\in M$. As in \textbf{Step~1} of Theorem~\ref{thm:intrinsic_metric_identifiable}, operator equality implies equality of the associated
integral operators $T_i$, hence the equality of finite Borel measures
\(
K_1(x,\cdot)\,dm_1 \;=\; K_2(x,\cdot)\,dm_2,
 dm_i:=p_i\,d\mu_i .
\)
Since $p_i>0$ and $\mu_i, m_i$ have full support, and $m_i$ are mutually absolutely continuous, write $dm_2=w\,dm_1$ with $w>0$ $m_1$-a.e.
Arguing exactly as in \textbf{Step~2} of Theorem~\ref{thm:intrinsic_metric_identifiable},
we obtain the pointwise identity
\(
K_1(x,y)=K_2(x,y)\,w(y) \forall x,y\in M.
\)
Setting $x=y$ yields $1=1\cdot w(y)$ for all $y$, hence $w\equiv 1$ and therefore $K_1\equiv K_2$. Consequently
\(
\|\iota_1(x)-\iota_1(y)\|=\|\iota_2(x)-\iota_2(y)\|\forall x,y\in M.
\)
By \textbf{Step~3} of Theorem~\ref{thm:extrinsic_induced_metric_identifiable}, $g_1=g_2$, hence $\mu_1=\mu_2$.
Finally, $dm_1=dm_2$ and $\mu_1=\mu_2$ imply $p_1=p_2$.
\end{proof}

\bibliographystyle{plainnat}   
\bibliography{references}      

\end{document}